\documentclass[12pt,twoside]{amsart} \usepackage{amssymb,color,tikz}
\usetikzlibrary{matrix} \nonstopmode \textwidth=16.00cm

\usepackage{tikz-cd}
\usepackage{lipsum}
\usepackage{adjustbox}
\usepackage{graphicx,color}
\usepackage{comment}
\usepackage{xcolor}

\numberwithin{equation}{section} 

\newcommand {\PP}{\mathbb{P}}

 \DeclareMathOperator{\Proj}{Proj}
 \def\cocoa{{\hbox{\rm C\kern-.13em
      o\kern-.07em C\kern-.13em o\kern-.15em A}}}
\newtheorem{theorem}{Theorem}[section]
\newtheorem{lemma}[theorem]{Lemma}
\newtheorem{proposition}[theorem]{Proposition}

\newtheorem{problem}[theorem]{Problem}
 \theoremstyle{definition}
\newtheorem{definition}[theorem]{Definition} \theoremstyle{remark}
\newtheorem{remark}[theorem]{Remark}
\newtheorem{example}[theorem]{Example}
\newtheorem{notation}[theorem]{Notation}

\DeclareMathOperator{\Ann}{Ann}
\DeclareMathOperator{\Hom}{Hom}
\DeclareMathOperator{\Tor}{Tor}

\DeclareMathOperator{\Ext}{Ext}

\newcommand{\cA}{{\mathcal A}}

\usepackage{hyperref}
\definecolor{MyDarkGreen}{cmyk}{0.7,0,1,0}

\begin{document}

\title{Betti numbers of full Perazzo algebras}

  \author[R.\ M.\ Mir\'o-Roig]{Rosa M.\ Mir\'o-Roig}
  \address{Facultat de
  Matem\`atiques i Inform\`atica, Universitat de Barcelona, Gran Via de les
  Corts Catalanes 585, 08007 Barcelona, Spain} \email{miro@ub.edu,  ORCID 0000-0003-1375-6547}

  \author[Josep Pérez]{Josep Pérez}
 \address{Facultat de
  Matem\`atiques i Inform\`atica, Universitat de Barcelona, Gran Via de les
  Corts Catalanes 585, 08007 Barcelona, Spain} \email{jperezdiez@ub.edu}


\begin{abstract}
In this  paper we prove that any full Perazzo algebra $A_F$, whose Macaulay dual generator is a Perazzo form $F\in K[X_0,\dots,X_n,U_1,\dots,U_m]_d$ with $n+1 = \binom{d+m-2}{m-1}$, is the doubling of a 0-dimensional scheme in $\PP^{n+m}$ and we compute the graded Betti numbers of a minimal free resolution of $A_F$.
\end{abstract}

\maketitle

\section{Introduction}
Artinian graded Gorenstein algebras  (also known as Poincar\'e duality algebras) have long been studied due to their ubiquity in many fields of mathematics, including algebraic geometry, commutative algebra, algebraic topology, combinatorics, etc. Over the last several decades, a general aspect of such algebras that
has been extensively studied is the existence of a structure theorem for their minimal free resolutions. For the codimension 2 case the result is well known since artinian Gorenstein algebras and complete intersection algebras coincide. For codimension 3 artinian Gorenstein algebras  there is the Buchsbaum and Eisenbud 
structure theorem \cite{BE}; for codimension $>3$
the matter becomes much more intriguing and no general structure theorem is known apart from particular cases like  codimension 4 artinian Gorenstein algebras defined by the submaximal minors of a $t\times t$ homogeneous matrix \cite{Gu}. In this paper, we will concentrate our attention on the problem of determining a minimal free resolution and, in particular, the graded Betti  numbers of certain artinian Gorenstein algebras which are the so-called full Perazzo algebras.

\vskip 2mm
Among one of the most important kinds of artinian Gorenstein algebras that have been recently deeply studied from the viewpoint of
Lefschetz properties, we have the Perazzo algebras. These algebras are associated via Macaulay-Matlis duality to 
a Perazzo hypersurface of degree $d$. By definition, a Perazzo hypersurface $X\subset \PP^{n+m}$ is a hypersurface defined by a degree $d$ form  $F\in K[X_0, \ldots,X_n,U_1,\ldots,U_m]$ of the following type  (see \cite{G}):
\[
F=X_0p_0+X_1p_1+\dots+X_np_n+G,
\]
with $n\geq m\geq 2$, $p_0, \ldots, p_n\in K[U_1, \ldots, U_m]_{d-1}$,  $G\in K[U_1, \ldots, U_m]_d$, where $p_0, \ldots, p_n$ are algebraically dependent but linearly independent.

\vskip 2mm
The interest in Perazzo hypersurfaces comes from the fact that they are a class of examples of hypersurfaces with vanishing hessian which are not cones. The characterization of such hypersurfaces is an important open problem in projective algebraic geometry since the classical works of Hesse and of Gordan-Noether (\cite{He1}, \cite{He2}, \cite{GN}). Hesse believed that the determinant of the hessian of a homogeneous polynomial  vanishes if and only if its variety of zeros  is
a cone, as is indeed the case when the degree of the form is $2$. However, P. Gordan and M. Noether 
proved that while Hesse’s claim is true for forms in at most $4$ variables, it is false for $5$ or more variables and any
degree $\geq 3$. Since then, many efforts have been made to find a characterization for arbitrary degree and number of variables. The first counterexamples known are precisely Perazzo hypersurfaces (see \cite[ Appendix A]{G}, \cite[Chapter 7]{R}), but there are other kinds of examples too (see \cite{G2}).

In \cite[Problem 2]{MM}, Mezzetti and the first author   posed the problem of describing the minimal free resolution of {\em any} 
Perazzo algebra. In general, this  turns out to be a
very difficult problem and so far few cases are known (see, for instance, \cite{JMR}, \cite{MM} and \cite{MR}). In this paper we determine the minimal free resolution of a large class of Perazzo algebras, namely, we determine the graded Betti numbers of {\em any} full Perazzo algebra (see Definition \ref{PerazzoDefinition}).
It is well known that the  minimal graded free $R$-resolution of $A_F$ has the following shape:
$$
0\longrightarrow \oplus _jR(-j)^{\beta_{n+1,j}^R(A_F)}
\longrightarrow \oplus _jR(-j)^{\beta_{n,j}^R(A_F)}\longrightarrow \cdots  \longrightarrow \oplus _jR(-j)^{\beta_{i,j}^R(A_F)} \longrightarrow
$$
$$\cdots \longrightarrow \oplus _jR(-j)^{\beta_{1,j}^R(A_F)}\longrightarrow  R\longrightarrow A_F\longrightarrow 0 
$$
and the graded Betti numbers $\beta_{i,j}^R(A_F)$ of $A_F$ over $R$ are defined as usual as  the integers 
$$\beta_{i,j}^R(A_F)=\dim_K [\Tor^R_i(A_F,K)]_j.$$ These homological invariants are our main focus. It is important to point out that the graded Betti numbers of an Artinian graded algebra $A_F$ determine its  graded Poincar\'{e} series. In fact, the graded Poincar\'{e} series of $A_F$ over $R$ is the generating function $$P^R_{A_F}(t,s)=\sum_{i,j} \beta_{i,j}^R(A_F) t^is^j.$$

\vskip 2mm
As a main tool we will use the doubling construction, which, together with the mapping cone process, yields a minimal free resolution of any full Perazzo algebra, by knowing a minimal free resolution of a certain 0-dimensional scheme. Having a structure theorem for the minimal free resolution of a full Perazzo algebra $A_F$ is interesting by itself, but also  because the Lefschetz properties of $A_F$  can be determined from its graded Betti numbers, by narrowing the possible Jordan types of the multiplication map $\times \ell :A_F\rightarrow A_F$, where $\ell$ represents a linear form.

We now give a more precise description of the main contents and the structure of this paper. In section \ref{prelim}, we recall the doubling construction, as well as the basic notions of Perazzo algebras and minimal free resolutions. In section \ref{toy}, we explain with details a toy example with the idea that it will help the reader to understand the strategy we use in order to prove the main results of this paper. Section \ref{MFRPerazzo} is the core of this paper and  contains the main result, i.e., the formula for the graded Betti numbers of a full Perazzo algebra. Finally, we state some open problems.

\section{Notation and background material}\label{prelim}

In this section we fix notation and we recall the basic facts on minimal free resolutions, {\em doublings} as well as on Perazzo algebras needed later on.

Throughout this paper  $K$ will be an algebraically closed field of characteristic zero.
Given a standard graded artinian $K$-algebra $A=R/I$ where $R=K[x_0,\dots,x_N]$ and $I$ is a homogeneous ideal of $R$,
we denote by $HF_A\colon \mathbb{Z} \longrightarrow \mathbb{Z}$ with $HF_A(j)=\dim_K[A]_j=\dim _K[R/I]_j$
its Hilbert function. Since $A$ is artinian, its Hilbert function is
captured in its $h$-\emph{vector} $h=(h_0,\dots ,h_d)$ where $h_i=HF_A(i)>0$ and $d$ is the largest index with this property. The integer $d$ is called the \emph{socle degree of} $A$ or the regularity of $A$ and denoted by $reg(A)$.


\subsection{Artinian Gorenstein algebras and Macaulay-Matlis duality}

In this subsection  we recall the definition of artinian Gorenstein algebra and we quickly recall the construction of the standard graded artinian Gorenstein algebra $A_F$ with {\em Macaulay dual generator} a given form $F\in S=K[X_0,\dots,X_N]$.

\begin{definition}
     A standard graded artinian algebra
$A = R/I$ is {\it Gorenstein} if its h-vector is symmetric and its socle is one-dimensional, occuring in just the last degree. We sometimes refer to the number of variables of $R$ as the {\it codimension} of $A$.
\end{definition}

Let $R=K[x_0,\ldots,x_N]$ be a polynomial ring and let $S=K[X_0,\ldots,X_N]$ be a divided power algebra, regarded as a $R$-module with the contraction action \cite[Example A.5, Appendix A]{IK}
\[
x_i\circ X_0^{j_0}..X_i^{j_i}..X_N^{j_N}=\begin{cases}X_0^{j_0}..X_i^{j_i-1}..X_N^{j_N} & \text{if} \ j_i>0\\ 0 & \text{otherwise}.\\ \end{cases}
\]

For each degree $i\geq 0$, the action of $R$ on $S$ defines a non-degenerate $K$-bilinear pairing 
\begin{equation}
\label{eq:MDPairing}
    R_i \times S_i \longrightarrow K \text{ with } (f,F) \longmapsto f \circ F.
\end{equation}
This implies that for each $i\geq 0$ we have an isomorphism of $K$-vector spaces $S_i\cong \Hom_K(R_i,K)$ given by $F\mapsto\left\{f\mapsto f\circ F\right\}$.

It is a classical result of Macaulay \cite{Mac} (cf. \cite[Theorem 2.1]{IK}) that an artinian $K$-algebra $A=R/I$ is Gorenstein with socle degree $d$ if and only if  $I=\Ann_R(F)=\{f\in R\mid f\circ F=0\}$ for some homogeneous polynomial $F\in S_d$.  Moreover, this polynomial, called a {\em Macaulay dual generator} for $A$, is unique up to a scalar multiple. 


\subsection{Perazzo algebras}

The simplest known counterexample to Hesse's claim, i.\,e.\ a homogeneous polynomial with vanishing  Hessian which does not define a cone, is $XU^2+YUV+ZV^2$. This example was extended to a class of cubic counterexamples in all dimensions  by Perazzo in \cite{P}. The results of Gordan-Noether and of Perazzo have been recently considered and rewritten in modern language by many authors  \cite{CRS}, \cite{L}, \cite{GR}, \cite{Wa} and \cite{WB}. Following these papers we define:

\begin{definition}\label{perazzohypersur} Fix integers $n, m\ge 2$. A {\em Perazzo} hypersurface  $X\subset \PP^{n+m}$ of degree $d\ge 3$ is a hypersurface defined by a form $F\in K[X_0,\dots ,X_n,U_1\dots ,U_m]_d$ of the following type:
$$ F=X_0p_0+X_1p_1+\cdots +X_np_n+g
$$
where  $p_i\in K[U_1,\dots ,U_m]_{d-1}$ are algebraically dependent but linearly independent and $g\in K[U_1,\dots ,U_m]_{d}$. By abuse of terminology, we call $F$ a {\em Perazzo} form.
\end{definition}

The fact that the $p_i$'s are algebraically dependent implies the vanishing of the hessian of $F$ and  $n+1 \le \binom{d+m-2}{m-1}$ (see \cite[Propositions 4.1 and 4.2]{WB}), while the linear independence assures that $V(F)$ is not a cone.

\begin{definition}\label{PerazzoDefinition}
    A {\em Perazzo} algebra $A_F$ is a standard graded artinian Gorenstein algebra whose Macaulay dual generator is a Perazzo form $F$ of degree $d\ge 3$. A Perazzo algebra $A_F$ is said to be {\em full} if  $F=X_0p_0+\cdots + X_np_n  \in K[X_0,\dots,X_n,U_1,\dots,U_m]_d$ and $\{p_0,\dots,p_n\}$ is a basis of $K[U_1,\dots,U_m]_{d-1}$. So, for full Perazzo algebras we necessarily have $n+1=\binom{d+m-2}{m-1}$.
\end{definition}

Lefschetz properties have been the subject of intense research in recent years, many natural problems are still open and the
general picture is far from being understood. The study of hessians of homogeneous polynomials has gained
new attention because of its connection to Lefschetz properties for
graded artinian Gorenstein algebras. The question if these algebras satisfy or fail the weak Lefschetz property has been
considered in some recent articles (\cite{A}, \cite{FMMR}, \cite{MM} and \cite{MP}), where the case of
Perazzo forms with $m = 2$ has been completely solved.

\subsection{Minimal free resolutions and graded Betti numbers}
 Let $M$ be a finitely generated  $R$-module. It is well known that it has a minimal graded free $R$-resolution of the following type:
$$
0\longrightarrow F_{N+1} 
\longrightarrow F_N \longrightarrow \cdots  \longrightarrow F_i \longrightarrow \cdots \longrightarrow F_1\longrightarrow  F_0\longrightarrow M\longrightarrow 0 
$$
where
$$ F_i=\oplus _jR(-j)^{\beta_{ij}^R(M)}
$$
and the graded Betti numbers $\beta_{ij}^R(M)$ of $M$ over $R$ are defined as usual as  the integers
$$
\beta_{i,j}:=\beta_{ij}^R(M)=\dim_K [\Tor^R_i(M,K)]_j.
$$ 

These homological invariants are our main focus and indeed our goal in Section \ref{MFRPerazzo} is to determine the graded Betti numbers $\beta _{ij}^R(A_F)$ of {\em any} full Perazzo algebra $A_F$ (see Theorem \ref{mainTheoremGeneralFullPerazzos}).

 A compact way to display the graded Betti numbers $\beta _{i,j}$  is the {\em Betti} table, that is,

    \vskip 4mm
    
    \begin{center}
    \begin{tabular}{ r | c c c c}
           &  0 &  1 &  $\cdots$ & N+1 \\
    \hline
     $\vdots$ &  $\vdots$ &  $\vdots$ & $\ddots$ & $\vdots$ \\
        k &  $\beta_{0,k}$ & $\beta_{1,k+1}$ & $\cdots$ &  $\beta_{N,k+N+1}$ \\  
        k+1 &  $\beta_{0,k+1}$ & $\beta_{1,k+2}$ & $\cdots$ & $\beta_{N,k+1+N+1}$ \\
 $\vdots$ &  $\vdots$ &  $\vdots$ & $\ddots$ & $\vdots$ \\
    \end{tabular}
    \end{center}
Note that in the $k$-th row of the $i$-th column there is $\beta_{i,k+i}$, instead of $\beta_{i,k}$, for a better visualisation of $reg(M)$, which can be computed using the graded Betti numbers as

$$reg(M) := max_{i,j}\{j - i | \beta_{i,j}(M) \ne 0\}.$$

Moreover, if $I$ is an artinian ideal, then $reg(R/I) = max_i\{i \ | \ \mathfrak{m}^{i+1} \subset I\}=socle(R/I)$ where $\mathfrak{m}$ denotes the irrelevant ideal.

\begin{example}\label{exampleMinimalFreeResolution}
    Consider the complete intersection $I=(x-y,y^2,z^3)\subset R:=K[x,y,z]$. The minimal free resolution of $I$ is given by the Koszul complex:
    $$
    0 \rightarrow R(-6) \xrightarrow{\begin{pmatrix}z^3 \\ -y^2 \\ x-y \end{pmatrix}} \begin{array}{c} R(-3) \\ \oplus \\ R(-4) \\ \oplus \\ R(-5) \end{array} \xrightarrow{\begin{pmatrix} -y^2 & -z^3 & 0 \\ x-y & 0 & -z^3 \\ 0 & x-y & y^2 \end{pmatrix}} \begin{array}{c} R(-1) \\ \oplus \\ R(-2) \\ \oplus \\ R(-3) \end{array} \xrightarrow{\begin{pmatrix} x-y & y^2 & z^3 \end{pmatrix}} I \rightarrow 0
    $$
\end{example}
Therefore, $reg(R/I)=3$, the $h$-vector of $R/I$ is $(1,2,2,1)$ and the Betti table is:

    \begin{center}
    \begin{tabular}{ r | c c c c}
           &  0 &  1 &  2 & 3 \\
    \hline
        0 & 1 & 1 & · & · \\  
        1 & · & 1 & 1 & · \\
        2 & · & 1 & 1 & · \\
        3 & · & · & 1 & 1 
    \end{tabular}
    \end{center}


\subsection{Doubling construction}
 Let us recall the doubling construction and some basic results related to it.

\begin{definition}
 The \emph{canonical module} of a graded $R$-module $M$ is defined:
 $$\omega_M := \Ext^{N+1 - \dim M}_R (M, R).$$ 
\end{definition}

\begin{definition}\label{doubling} \cite[Section 2.5]{KKRSSY} Let $J\subset R$ be a homogeneous ideal of codimension $c$, such that $R/J$ is Cohen-Macaulay and $\omega_{R/J}$ is its canonical module. Furthermore, assume that $R/J$ satisfies the condition $G_0$ (i.e., it is Gorenstein at all minimal
primes). 
Let $I$ be an ideal of codimension $c + 1$. $I$ is called a {\em doubling } of $J$ via $\psi$ if there exists  a short exact sequence of $R/J$ modules
\begin{equation}\label{eq:doubling}
0 \rightarrow \omega_{R/J}(-d)\stackrel{\psi}{\rightarrow} R/J \rightarrow R/I\rightarrow 0.
\end{equation}
We will also say, by abuse of terminology, that $R/I$ is the doubling of $\Proj(R/J)$.
\end{definition}

By \cite[Proposition 3.3.18]{BH}, if $I$ is a doubling then $R/I$ is a Gorenstein ring. Doubling plays an important role in the theory of Gorenstein liaison. Indeed, in \cite{KMMNP}, doubling is used to produce suitable Gorenstein divisors on arithmetically Cohen-Macaulay subschemes in several foundational constructions.
It is not true that every
artinian Gorenstein ideal of codimension $c + 1$ is a doubling of some codimension $c$ ideal (see, for instance, 
\cite[Example 2.19]{KKRSSY}).

Moreover, the mapping cone of $\psi$ in \eqref{eq:doubling} gives a resolution of $R/I $. If it is minimal, then
one can read off the Betti table of $R/I$ from the Betti table of $R/J$. This mapping cone is the direct sum of the minimal free resolution $F_{\bullet }$ of $R/J$ with its dual (reversed) complex $\Hom(F_{\bullet},R)$, which justifies the terminology of {\em doubling}.

\begin{example}
    Consider the ideal $J=(y^2,z^3)\subset R:=K[x,y,z]$. The minimal free resolution of $R/J$ is:
    $$
    0 \rightarrow R(-5) \xrightarrow{\begin{pmatrix}-z^3 \\ y^2 \end{pmatrix}} \begin{array}{c} R(-2) \\ \oplus \\ R(-3) \end{array} \xrightarrow{\begin{pmatrix} y^2 & z^3 \end{pmatrix}}R \rightarrow R/J \rightarrow 0
    $$
    Dualizing and twisting by $-6$ we get the minimal free resolution of $\omega_{R/J}(-3)$. That is:
    $$
    0 \rightarrow R(-6) \rightarrow \begin{array}{c} R(-3) \\ \oplus \\ R(-4) \end{array} \rightarrow R(-1) \rightarrow \omega_{R/J}(-3) \rightarrow 0
    $$
    Then, if we send the only generator of $\omega_{R/J}(-3)$, which has degree $1$, to the generator $x-y$ of the ideal $I$ in Example \ref{exampleMinimalFreeResolution}, we obtain the following short exact sequence:
    $$
    0 \rightarrow \omega_{R/J}(-3) \xrightarrow{\psi} R/J \rightarrow R/I \rightarrow 0
    $$
    Therefore, $I$ is the doubling of $J$. Notice that we can recover the graded Betti numbers of $R/I$ from the mapping cone of $\psi$.
\end{example}

\vskip 4mm

\section{Toy example}\label{toy}


The goal of this section is to ease the comprehension of the main results in this paper, by studying  a toy example to illustrate all of them. Indeed, we consider  a full Perazzo algebra $A_F$, with socle degree $d+1$, whose Macaulay dual generator $F$ lies in $R=K[X_0,\dots,X_d,U,V]$. We will prove that $A_F$ is the doubling of a suitable  0-dimensional scheme $Z_F\subset \PP^{d+2}=\Proj(R)$ of length $2(d+1)$ and, using the mapping cone process, we will be able to compute its graded Betti numbers.

\begin{lemma}\label{resolutionFromHVector}
    Let $A=R/I$ be an artinian algebra of codimension $c$, socle degree $e\ge 3$ and $h$-vector $(1,c,1,\cdots ,1)$. It holds:

    (1) Up to change of coordinates: $I=\langle x_{c-1}^{e+1}, x_ix_j \mid 0\le i,j\le c-1 \text{ and } (i,j)\ne (c-1,c-1) \rangle $
     
     (2) The Betti table of $A$ is:
    \begin{center}
    \begin{tabular}{ r | c c c c c}
           &  0 &  1 &  2 & $\cdots$ &  c \\
    \hline
        0 &  1 &  · &  · &  · &  · \\  
        1 &  · & $\beta_{1,2}$ & $\beta_{2,3}$ & $\cdots$ & $\beta_{c,c+1}$ \\
        . &  · &  · &  · &  · &  · \\
      e &  · &  $\beta_{1,e+1}$ &  $\beta_{2,e+2}$ &  $\cdots$ &  $\beta_{c,e+c}$ \\    
    \end{tabular}
    \end{center}
where 
$$
\begin{cases}
    \beta_{i,e+i}=\binom{c-1}{i-1} \text{ for } 1\le i \le c; \\\\
    \beta_{i,i+1}=i\binom{c+1}{i+1}-\binom{c-1}{i-1} \text{ for } 1\le i \le c.
\end{cases}
$$
\end{lemma}

\begin{proof}
    (1) Set $R:=K[x_0,\ldots,x_{c-1}]$ and let $S:=K[X_0,\ldots,X_{c-1}]$ be a divided power algebra, regarded as a $R$-module with the contraction action. Applying Macaulay-Matlis duality and taking into account that  $A$ is an artinian algebra of codimension $c$, socle degree $e\ge 3$ and $h$-vector $(1,c,1,\cdots ,1)$ we get (after changing coordinates, if necessary):
    $$I=\Ann _R(M)=\{f\in R \mid f\circ F=0 \text{ for all } F\in M\}$$
    where $M=\{X_{c-1}^e,X_0,\cdots ,X_{c-1}\}$ and the result follows.

    (2) Consider the algebra $B=R/J_2$ where $J_2$ is the ideal generated by the quadrics in $I$.  We have $HF_B(t)=HF_A(t)$ for $t\le e$ and $HF_B(t)=1$ for all $t\ge 2$. Therefore, the t-th rows, $0\le t \le e-1$, of the Betti tables of $A$ and $B$ coincide. But all those rows are zero in the Betti table of $B$ except for the 1-st row, whose Betti numbers are: 
    $$
    \beta_{i,i+1}=i\binom{c+1}{i+1}-\binom{c-1}{i-1} \text{ for } 1\le i \le c.
    $$
    Now, since $reg(A)=e$, the e-th row must be the last non-zero row and its corresponding Betti numbers are completely determined by the $h$-vector of $A$.    
\end{proof}

\begin{notation}\label{doublePointNotation}   
Set $R:=K[x_0,\dots,x_n,u,v]$ and $S:=K[X_0,\dots,X_n,U,V]$. Let $\ell_i=\lambda_i u+\mu_i v$, $0\le i \le n$,  be linear forms two by two linearly independent and define $q_i\in \Proj(R)=\PP^{n+2}$ to be the double point with homogeneous ideal $I_i=(x_0,\dots,x_i^2,\dots,x_n,\ell_i) $
\end{notation}

 \begin{proposition}\label{zeroSchemeResolutionViaArtinianReduction} With the above notation, let $n\ge 2$ and  let $Z=\{q_0,\cdots ,q_n\}\subset \PP^{n+2}$ be the 0-dimensional scheme defined by $I:=\cap _{i=0}^nI_i$. Then the following holds:
 \begin{itemize}
     \item[(1)] $deg(Z)=2(n+1)$,
     \item[(2)] The $h$-vector of $Z$ is $(1,n+2,1,...,1)$, and
     \item[(3)] The Betti table of $R/I(Z)$ looks like:
 \begin{center}
    \begin{tabular}{ r | c c c c c}
           &  0 &  1 &  2 & $\cdots$ &  n+2 \\
    \hline
        0 &  1 &  · &  · &  · &  · \\  
        1 &  · & $\beta_{1,2}$ & $\beta_{2,3}$ & $\cdots$ & $\beta_{n+2,n+3}$ \\
        · &  · &  · &  · &  · &  · \\
      n &  · &  $\beta_{1,1+n}$ &  $\beta_{2,2+n}$ &  $\cdots$ &  $\beta_{n+2,2n+2}$ \\    
    \end{tabular}
    \end{center}
 where $$
\beta_{i,j}=
\begin{cases}
    \beta_{i,n+i}=\binom{n+1}{i-1} \text{ for } 1\le i \le n+2; \\\\
    \beta_{i,i+1}=i\binom{n+3}{i+1}-\binom{n+1}{i-1} \text{ for } 1\le i \le n+2.
\end{cases}
$$
 \end{itemize}
    
\end{proposition}
\begin{proof}
    (1) Since $Z$ is a 0-dimensional scheme consisting of $n+1$ double points we have $\deg(Z)=2(n+1)$. \\
    (2) Note that $I$ is generated by:
    \begin{itemize}
        \item $x_ix_j$ for all $i,j \in \{0,1,\dots,n\}$;
        \item $x_0\ell_0$, $x_1\ell_1$, \dots, $x_n\ell_n$;
        \item $\ell_0\cdot \ell_1\cdot...\cdot \ell_n$.
    \end{itemize}
    Hence, $HF_{R/I}(1)=n+3$. Moreover, since all the generators of $I$ have degree 2 except for one of degree $n+1$, the Hilbert function $HF_{R/I}(t)$ strictly increases up until $t=n$, where it reaches the constant value of the Hilbert polynomial, that is, $HF_{R/I}(t)=2n+2$ for all $t\ge n$. In consequence, the only possible Hilbert function  $HF_{R/I}(t)$ for $R/I$ is
    $$
    (1,n+3,n+4,\dots,2n+2,2n+2,\dots).
    $$
    Therefore, the $h$-vector of its artinian reduction  is given by $$\Delta HF_{R/I}(t)=HF_{R/I}(t)-HF_{R/I}(t-1)=(1,n+2,1,\dots,1).$$
    (3) The case $n=2$ is immediate. The general case  follows from Lemma \ref{resolutionFromHVector} (with $c=n+2$ and $e=n$), taking into account that the Betti table of $R/I$ coincides with the Betti table of its artinian reduction.
\end{proof}

\begin{theorem}\label{doublingPerazzoFull} Fix an integer $d\ge 2$.  Any full Perazzo algebra $A_F$  with dual generator $F\in K[X_0,..,X_d,U,V]$ and  socle degree $d+1$ is a doubling of a 0-dimensional scheme $Z_F\subset \PP^{d+2}$  consisting of $d+1$ double points on a line.
\end{theorem}

\begin{proof} We first observe that,  after a change of coordinates, the dual generator $F$ of a full Perazzo  algebra $A_F$ can be expressed $F=X_0L_0^d+\cdots +X_dL_d^d$ where $L_i\in K[U,V]_1$ are linear forms two by two linearly independent. Actually, this follows from the fact that $L_0^d,\cdots ,L_d^d$, with $L_i=\alpha _iU+\gamma_iV$ general linear forms, is a basis of $K[U,V]_d$.

We will now compute a minimal set of generators of $\Ann_R(F)$. We first observe that the following set $\cA$ of polynomials  is part of a minimal set of generators of $\Ann_R(F)$:
    \begin{itemize}
  \item $x_ix_j$ for all $i,j \in \{0,1,\dots,d\}$;
        \item $x_0\ell_0$, $x_1\ell_1$, \dots, $x_d\ell_d$ with $\ell _i=\gamma_iu-\alpha_iv$;
        \item $u^{d+1}$, $u^dv$,\dots,$uv^d$,$v^{d+1}$.
           \end{itemize}
           Now, from \cite[cf. proof Theorem 3.4]{MP} we know that the $h$-vector of $A_F$ is 
           $$
           (1,d+3,d+3,\cdots ,d+3,d+3,1).
           $$ 
           Therefore, we conclude that there is a unique quadric $q\in K[x_0,\cdots ,x_d,u,v]$ such that $\cA \cup\{q\}$ is a minimal system of generators of $\Ann_R(F)$. 

We consider now the 0-dimensional subscheme $Z_F\subset \PP^{d+2}$ with homogeneous ideal $I(Z_F)=\cap _{i=0}^d(x_0,\cdots , x_i^2,\cdots,x_d,\ell _i)$. It holds $I(Z_F)\subset \Ann_R(F)$ and, using Proposition \ref{zeroSchemeResolutionViaArtinianReduction},  we will prove that $A_F$ is a doubling of $Z_F$. 

We first observe that $\dim [R/I(Z_F)]_2=\dim  [A_F]_2+1$ and that $\Ann_R(F)$ has $d+1$ more generators of degree $d+1$ than $I(Z_F)$ does. On the other hand, by Proposition \ref{zeroSchemeResolutionViaArtinianReduction} we know a minimal free resolution of $R/I(Z_F)$:
$$0\rightarrow \begin{array}{c}
R(-d-3)^{\beta_{d+2,d+3} }\\
\oplus \\
R(-2(d+1))^{\beta_{d+2,2(d+1)}}
\end{array}
\rightarrow \cdots  
\rightarrow \begin{array}{c}
R(-2)^{\beta_{1,2}} \\
\oplus \\
R(-d-1)^{\beta_{1,d+1}}
\end{array}
\rightarrow 
I(Z_F) \rightarrow 0
$$
 where $$
\beta_{i,j}=
\begin{cases}
    \beta_{i,d+i}=\binom{d+1}{i-1} \text{ for } 1\le i \le d+2; \\\\
    \beta_{i,i+1}=i\binom{d+3}{i+1}-\binom{d+1}{i-1} \text{ for } 1\le i \le d+2.
\end{cases}
$$
as well as a minimal free resolution of the canonical module $\omega_{R/I(Z_F)}$ of $R/I(Z_F)$:
$$
0\rightarrow  R(-d-3) \rightarrow 
 \begin{array}{c}
R(-2)^{\beta_{1,d+1} }\\
\oplus \\
R(-d-1)^{\beta_{1,2}}
\end{array}
\rightarrow \cdots \rightarrow
\begin{array}{c}
R(d-1)^{\beta_{d+2,2(d+1)}} \\
\oplus \\
R^{\beta_{d+2,d+3}}
\end{array}
\rightarrow 
\omega _{R/I(Z_F)}\rightarrow 0.
$$
So, $\omega _{R/I(Z_F)}(-d-1)$ has $\beta_{d+2,d+3}=d+1$  generators of degree $d+1$ and one generator of degree $2$. Thus, since we know a minimal system of generators of $\Ann_R(F)$, we can construct the exact sequence
    \begin{equation}\label{doublingSequence}
        0 \rightarrow \omega_{R/I(Z_F)}(-d-1) \xrightarrow{\psi} R/I(Z_F) \rightarrow R/\Ann_R(F) \rightarrow 0
    \end{equation}
sending the generator of degree $2$ of  $\omega _{R/I(Z_F)}(-d-1)$ to $q$ and the $d+1$ generators of degree $d+1$ to the different classes of $u^{d+1}$, $u^dv$,\dots,$uv^d$,$v^{d+1}$ modulo $I(Z_F)$.
The  map $\psi $ is well defined since these elements satisfy the
same relations that the generators of $\omega_{R/I(Z_F)}(-d-1)$ and it is injective. Therefore, we conclude that $A_F$ is a doubling of $Z_F$.
\end{proof}

As an application we get:

\begin{theorem}\label{mainResult0} Fix an integer $d\ge 2$. The Betti table of a full Perazzo algebra $A_F$ associated to a form $F\in K[X_0,\cdots, X_d,U,V]_{d+1}$ is:

\vskip 4mm
 \begin{center}
    \begin{tabular}{ r | c c c c c c}
           &  0 &  1 &  2 & $\cdots$ &  d+2 & d+3 \\
    \hline
        0 &  1 &  · &  · &  · &  · & .\\  
        1 &  · & $\alpha_{2}$ & $\alpha_{3}$ & $\cdots$ & $\alpha_{d+3}$ & . \\
        · &  · &  · &  · &  · &  · & · \\
        d & · & $\alpha_{d+3}$ & $\cdots$ & $\alpha_{3}$ & $\alpha_{2}$ & . \\
        d+1 &  · &  · &  · &  · &  · & 1\\
    \end{tabular}
    \end{center}
    where $\alpha_i=(i-1)\binom{d+3}{i}$ for $2\le i \le d+3$.
\end{theorem}

\begin{proof} 
    According to Theorem \ref{doublingPerazzoFull}, $A_F$ is a doubling of a 0-dimensional scheme $Z_F\subset \PP^{d+2}$  consisting of $d+1$ double points on a line. Using the mapping cone process in \ref{doublingSequence} one  computes a  minimal free resolution of $A_F$. Since $A_F$ is a Gorenstein ring, its minimal free resolution is self-dual up to twist. Therefore, we only need to compute the graded Betti numbers of the $1$-st non-zero row of its Betti table. For $2\le i \le d+3$, let $\beta_{i-1,i}$ and $\beta_{i-1,d+i-1}$ be the graded Betti numbers of the $1$-st and $d$-th rows of the Betti table of $R/I(Z_F)$, respectively. Then,
    $$\begin{array}{rcl}
    \alpha_i & = & \beta_{i-1,i} + \beta_{d+4-i,2d+4-i} \\\\
    & = &  (i-1)\binom{d+3}{i} - \binom{d+1}{i-2} + \binom{d+1}{d+3-i} \\\\
    & = & (i-1)\binom{d+3}{i}
    \end{array},$$  
concluding the proof.
\end{proof}

\begin{example} 
We consider a full Perazzo algebra $A_F$ with dual generator  $F=X_0U^3+X_1V^3+X_2(U+V)^3+X_3(U+3V)^3\in S_4$. By Theorem \ref{doublingPerazzoFull}, $A_F$ is the doubling of the 0-dimensional subscheme $Z_F\subset \PP^5$ with homogeneous ideal $$I(Z_F)=(x_0^2,x_1,x_2,x_3,v)\cap (x_0,x_1^2,x_2,x_3,u)\cap (x_0,x_1,x_2^2,x_3,u-v)\cap (x_0,x_1,x_2,x_3^3,3u-v)$$ 

\noindent and   the Betti table of $R/\Ann_R(F)$ is:

    \begin{center}
    \begin{tabular}{ r | c c c c c c c}
           &  0 &  1 &  2 & 3 & 4 & 5 & 6 \\
    \hline
        0 &  1 & · &  · &  · &  · &  . & · \\  
        1 &  · & 15 & 40 & 45 & 24 & 5 & · \\
        2 &  · & · &  · &  · & · & . & · \\
        3 &  · & 5 & 24 & 45 & 40 & 15 & · \\
        4 &  · & · & · & · & · & · & 1
    \end{tabular}
    \end{center}
All examples have been checked using Macaulay2 \cite{M2}        
\end{example}

\section{Minimal free resolutions for full Perazzo algebras}\label{mainSection}
\label{MFRPerazzo}

The goal of this section  is to determine the graded Betti numbers of {\em any} full Perazzo algebra $A_F$ with dual generator $F\in K[X_0,\dots,X_n,U_1,\dots,U_m]$. Notice that, in this case, there are no such algebras for each $n$, because we need to impose $\binom{d+m-2}{m-1}=n+1$, where now $d$ represents the degree of the corresponding Macaulay dual generator $F$. As a first step we will prove
 that any full Perazzo algebra $A_F$ is the doubling of a  0-dimensional scheme $Z_F\subset \PP^{n+m}=\Proj(R)$ consisting of $n+1$ double points.

\vskip 4mm Let us start this section with some preliminary results  and with a couple of technical lemmas.

\begin{definition}
    Let $Mon_d(R)$ be the set of all monomials in $R=K[x_1,\cdots ,x_r]$ of degree $d$. For any monomial $u\in R$, we set $m(u):=max\{i:x_i \text{ divides } u\}$. A set $\mathcal{L}\subset Mon_d(R)$ is \textit{stable} if $x_i(u/x_{m(u)})\in \mathcal{L}$ for all $u\in \mathcal{L}$ and all $i< m(u)$. A monomial ideal $J:=(u_1,\dots,u_m)\subset R$ is called \textit{stable} if $x_i(u_j/x_k)\in J$ for all $j\in\{1,\dots,m\}$ and all $i<k$ such that $x_k$ divides $u_j$.
\end{definition}

\begin{example}Consider $R=K[x_1,x_2,x_3]$

    (1) The monomial  ideal $I=(x_1^3,x_1^2x_2,x_1^2x_3)\subset R$ is stable. 

    (2) The monomial  ideal $J=(x_1^3,x_1^2x_2,x_2^2x_3)\subset R$ is not stable because $x_1x_2^2 \not\in J$.
\end{example}

\begin{lemma}\label{h2GeneralFullPerazzo} Fix $d\ge 4$ and let $A_F$ be a full Perazzo algebra with Macaulay dual generator 

$$
F\in K[X_0,\dots ,X_n,U_1,\dots,U_m]_d.
$$
Then:
$$
HF_{A_F}(i)=
\begin{cases}
    \binom{i+m-1}{m-1} + \binom{d-i+m-1}{m-1} \quad \text{for } 1\le i \le \lfloor d/2 \rfloor \\
    \text{symmetry}
\end{cases} .
$$
\end{lemma}

\begin{proof}
\cite[Proposition 3.3]{CGIZ}.
\end{proof}

\begin{remark}
    If $d=3$ the previous formula makes no sense, when evaluated at $i=2$. But in this case the Hilbert function is $HF_{A_F}=(1,n+m+1,n+m+1,1)$, for any $n,m \ge 2$ satisfying the full Perazzo condition $\binom{d+m-2}{m-1}=n+1$.
\end{remark}

\begin{lemma}\label{linearSyzygiesIdealOfQuadrics} Fix integers $1\le n<r$.
    Let $R=K[x_1,\dots,x_r]$ and consider the ideal $J$ with the following minimal set of generators:
    \begin{itemize}
        \item $x_ix_j$ \quad where $i,j \in \{1,2,\dots,n\}$;
        \item $x_1x_{n+1},x_1x_{n+2},\dots,x_1x_{r},\dots,x_nx_{n+1},x_nx_{n+2},\dots,x_nx_{r}$.
    \end{itemize}
    Then, $R/J$ has a linear resolution and the graded Betti numbers are given by:
    $$
    \beta_{i,i+1}= i\binom{n+1}{i+1} + n\left( \binom{r}{i} - \binom{n}{i}\right).
    $$    

\end{lemma}

\begin{proof}
First of all, notice that $J$ is a stable ideal whose generators are all of degree two. Therefore, by \cite[Corollary 7.2.3]{HH}, it has a linear resolution:
\begin{equation}\label{mfr}
    0 \rightarrow R(-r-1)^{\beta_{r,r+1}} \rightarrow \cdots \rightarrow R(-3)^{\beta_{2,3}} \rightarrow R(-2)^{\beta_{1,2}}\rightarrow R \rightarrow R/J \rightarrow 0 
\end{equation}
Let us determine the graded Betti numbers  $\beta_{i,i+1}(R/J)$. To this end, we denote by  $m_k$ the number of generators in $J$ whose greatest subindex equals $k$. We have:

$$
m_k=\begin{cases}
    k \quad \text{if} \ k\le n \\
    n \quad \text{if} \ n < k \le r
\end{cases}
$$
Now the result follows from the well-known Eliahou-Kervaire formula for the Betti numbers (cf. \cite[Corollary 7.2.3]{HH})
$\beta_{i,i+1}(R/J)=\sum_{k=1}^r\binom{k-1}{i-1}m_k$.
\end{proof}

\begin{lemma}\label{intersectionOfIdeals}
    Let $R=K[u_1,\dots,u_m]$. Let $d\ge 3$ and consider a set $P$ of $\binom{d+m-2}{m-1}$ general points. Then the Hilbert function of its defining ideal $I$ is
    \begin{equation}\label{hf}
    HF_{R/I}(t)=\left( 1, m, \binom{m+1}{m-1},\binom{m+2}{m-1},\dots, \binom{d+m-2}{m-1},  \binom{d+m-2}{m-1}, \dots  \right).
    \end{equation}
Moreover, $I$  has only generators of degree $d$ and there are exactly $\binom{d+m-2}{m-2}$.
\end{lemma}

\begin{proof}
    Such a set of points $P$ imposes independent conditions on forms of degree $\le d-1$. Hence, the Hilbert function of $R/I$ is the cited one and $I$ has exactly $\binom{d+m-2}{m-2}$ generators of degree $d$. 
\end{proof}

From now on, we will say that a set $P$ of $\binom{d+m-2}{m-1}$ points in $\PP^{m-1}$ are in {\em general position} if its homogeneous ideal $I(P)\subset  R=K[u_1,\dots,u_m]$ is generated by $\binom{d+m-2}{m-2}$ forms of degree $d$ and the Hilbert function of $R/I(P)$ is given by (\ref{hf}).

\begin{proposition}\label{bettiNumbersIZF}
    Let $d\ge 3$ and consider, in the linear subspace $\PP^{m-1}$ defined by $x_0=x_1=\cdots =x_n=0$, the 0-dimensional subscheme $Z_F \subset \PP^{n+m}$ consisting of $n+1:=\binom{d+m-2}{m-1}$ double points, with homogeneous ideal 
    $$
    I(Z_F)=\bigcap_{i=0}^n\left(x_0,\dots,x_i^2,\dots,x_n,\ell_{i,2},\dots,\ell_{i,m}\right) \subset K[x_0,\dots,x_n,u_1,\dots,u_m]
    $$
    where $\ell_{i,k}:=\lambda_{i,k}u_1-u_k$ and, for any $k\in\{2,\dots,m\}$ and all $i,j\in\{0,\dots,n\}$, $\lambda_{i,k}\not=\lambda_{j,k}$. Let us also assume that the points 
    $$
    V\left(\bigcap_{i=0}^n(\ell_{i,2},\dots,\ell_{i,m})\right) \subset \PP^{m-1}
    $$
    are in general position. Then, the Betti table of $R/I(Z_F)$ looks like:
\begin{center}
    \begin{tabular}{ r | c c c c c}
           &  0 &  1 &  2 & $\cdots$ &  n+m \\
    \hline
        0 &  1 &  · &  · &  · &  · \\  
        1 &  · & $\beta_{1,2}$ & $\beta_{2,3}$ & $\cdots$ & $\beta_{n+m,n+m+1}$ \\
        . &  · &  · &  · &  · &  · \\
      d-1 &  · &  $\beta_{1,1+d-1}$ &  $\beta_{2,2+d-1}$ &  $\cdots$ &  $\beta_{n+m,n+m+d-1}$\\    
    \end{tabular}
    \end{center}
 where
 $$
 \beta_{i,j}  =\begin{cases}
    \beta_{i,d-1+i}=\sum_{j=0}^{m-2} \Bigl( \binom{j+n+1}{i-1}\sum_{\ell=0}^{d-1}\binom{\ell+j-1}{j-1}\Bigr)\\\\
    \beta_{i,i+1}= i\binom{n+2}{i+1} + (n+1)\Bigl( \binom{n+m}{i} - \binom{n+1}{i} \Bigr)
\end{cases}.
$$
\end{proposition}
\begin{proof}
    Note that $I(Z_F)$ is generated by:
    \begin{itemize}
        \item $x_ix_j$ for all $i,j \in \{0,1,\dots,n\}$;
        \item $x_0\ell_{0,2}$,\dots,$x_0\ell_{0,m}$, \dots, $x_n\ell_{n,2}$, \dots, $x_n\ell_{n,m}$;
        \item $g_1,\dots,g_{\binom{d+m-2}{m-2}} \in K[u_1,\dots,u_m]_{d}$. (Lemma \ref{intersectionOfIdeals})
    \end{itemize}
    So, since we already know the Hilbert function of $K[u_1,\dots,u_m]/\cap_{i=0}^n(\ell_{i,2},\dots,\ell_{i,m})$ (Lemma \ref{intersectionOfIdeals}), it is straightforward to check that the Hilbert function of $R/I(Z_F)$ is
    \begin{equation}\label{HFRI}
    HF_{R/I(Z_F)}=\left(1, m + n+1, \binom{m+1}{m-1} + n+1, \dots, 2n+2, 2n+2 \dots \right),
    \end{equation}
    where we used that $n+1=\binom{m+d-2}{m-1}$ and hence $2n+2=\binom{m+d-2}{m-1} +n+1$.
    Certainly, we are adding $n+1$ variables $x_0,\cdots ,x_n$ to the polynomial ring $K[u_1,\dots,u_m]$ and, by construction, for any integer $t\ge 1$ there are exactly $n+1$ forms of degree $t$ in $R/I(Z_F)$ which are not in $[K[u_1,\dots,u_m]/\cap_{i=0}^n(\ell_{i,2},\dots,\ell_{i,m})]_t$.
    
    Now, the $h$-vector of $Z_F$ can be computed by taking the first difference in (\ref{HFRI}), that is,
\begin{equation} \label{hvectorZ_F}
   \Delta HF_{R/I(Z_F)}= HF_{\hat{R}/J}=\left(1, n+m, \binom{m}{m-2},\binom{m+1}{m-2},\dots,\binom{d+m-3}{m-2}\right),
    \end{equation}
    where $\hat{R}:=R/(u_1)$ and $J$ is the ideal generated by the classes of:
    \begin{itemize}
        \item $x_ix_j$ for all $i,j \in \{0,1,\dots,n\}$;
        \item $x_0u_2$,\dots,$x_0u_m$,\dots, $x_nu_2$, \dots, $x_nu_m$;
        \item $u_2^{i_2}\cdot u_3^{i_3}\cdot ... \cdot u_m^{i_m}$ with $ i_2+\cdots+i_m=d$.
    \end{itemize}
    Let $J_2$ be the ideal generated by the quadrics in $J$ and consider $B=\hat{R}/J_2$. Then, $HF_{\hat{R}/J}(\nu)=HF_B(\nu)$ for $\nu\le d-1$ and the $\nu$-th rows, $0\le \nu \le d-2$, of the Betti tables of $\hat{R}/J$ and $B$ coincide. But all those rows are zero in the Betti table of $B$ except for the 1-st row (Lemma \ref{linearSyzygiesIdealOfQuadrics}). Thus
    $$
        \beta_{i,i+1}= i\binom{n+2}{i+1} + (n+1)\Biggl( \binom{n+m}{i} - \binom{n+1}{i} \Biggr).
    $$
    Moreover, since $reg(R/I(Z_F))=d-1$, the (d-1)-th row must be the last non-zero row in the Betti table of $R/I(Z_F)$ and its corresponding Betti numbers are completely determined, once renamed $x_0,\dots,x_n,u_2,\dots,u_m$ as $x_1,\dots,x_{n+1},x_{n+2},\dots,x_{n+m}$ respectively, via the general Eliahou-Kervaire formula for the Betti numbers (cf. \cite[Corollary 7.2.3]{HH})
    $$
        \beta_{i,i+d-1}(\hat{R}/J)=\sum_{k=1}^{n+m}\binom{k-1}{i-1}m_{k,d},
    $$
    where $m_{k,d}$ denotes the number of generators of degree $d$ in $J$ whose greatest subindex equals $k$. But those generators involve only the variables $x_{n+2},\dots,x_{n+m}$, so
    $$
    \beta_{i,i+d-1}(\hat{R}/J)=\sum_{k=n+2}^{n+m}\binom{k-1}{i-1}m_{k,d}=\sum_{j=0}^{m-2}\Biggl(\binom{j+n+1}{i-1}\sum_{\ell=0}^{d-1}\binom{\ell+j-1}{j-1}\Biggr),
    $$
    which proves what we want.
\end{proof}

\begin{theorem}\label{mainTheoremGeneralFullPerazzos} Let $A_F$ be any full Perazzo algebra with Macaulay dual generator $F\in K[X_0,\dots,X_n,U_1,\dots,U_m]_d$. Then the following holds:
\begin{itemize}
    \item [(1)] $A_F$ is a doubling of a 0-dimensional scheme $Z_F\subset \PP^{n+m}$, consisting of $n+1$ double points.
    \item[(2)]  The Betti table of  $A_F$ is:

\vskip 4mm
 \begin{center}
    \begin{tabular}{ r | c c c c c c}
           &  0 &  1 &  2 & $\cdots$ &  n+m & n+m+1 \\
    \hline
        0 &  1 &  · &  · &  · &  · & .\\  
        1 &  · & $\alpha_{2}$ & $\alpha_{3}$ & $\cdots$ & $\alpha_{n+m+1}$ & . \\
        · &  · &  · &  · &  · &  · & · \\
      d-1 & · & $\alpha_{n+m+1}$ & $\cdots$ & $\alpha_{3}$ & $\alpha_{2}$ & . \\
        d &  · &  · &  · &  · &  · & 1\\
    \end{tabular}
    \end{center}
    where
    \begin{align*}
        \alpha_i &= (i-1)\binom{n+2}{i} + (n+1)\Biggl(\binom{n+m}{i-1} - \binom{n+1}{i-1} \Biggr) \\
                 &+ \sum_{j=0}^{m-2} \Biggl( \binom{j+n+1}{n+m+1-i}\sum_{\ell=0}^{d-1}\binom{\ell+j-1}{j-1} \Biggr) \quad \text{for} \quad 2\le i \le n+m+1.
    \end{align*}
\end{itemize}
\end{theorem}

\begin{proof} 

(1)     First we observe that, up to a linear change of variables, the dual generator $F$ of a full Perazzo algebra can be written as follows:  $$F=X_0L_0^{d-1}+\cdots+X_nL_n^{d-1},$$ where $L_0^{d-1},\dots,L_n^{d-1}$ is a basis of $K[U_1,\dots,U_m]_{d-1}$,
    $$
    L_i= U_1 + \lambda_{i,2}U_2 \cdots + \lambda_{i,m} U_m
    $$
  and for any $k\in\{2,\dots,m\}$ and all $i,j\in\{0,\dots,n\}$, $\lambda_{i,k}\not=\lambda_{j,k}$.  In the linear subspace $\PP^{m-1}$ defined by $x_0=x_1=\cdots =x_n=0$, we consider the 0-dimensional subscheme $Z_F \subset \PP^{n+m}$ consisting of $n+1:=\binom{d+m-2}{m-1}$ double points, with homogeneous ideal 
    $$
    I(Z_F)=\bigcap_{i=0}^n\left(x_0,\dots,x_i^2,\dots,x_n,\ell_{i,2},\dots,\ell_{i,m}\right) \subset K[x_0,\dots,x_n,u_1,\dots,u_m]
    $$
    where $\ell_{i,k}:=\lambda_{i,k}u_1-u_k$. Remember that a minimal set of generators of $I(Z_F)$ is:
    \begin{itemize}
        \item $x_ix_j$ for all $i,j \in \{0,1,\dots,n\}$;
        \item $x_0\ell_{0,2},\dots,x_0\ell_{0,m}, \dots ,x_n\ell_{n,2},\dots,x_n\ell_{n,m}$;
        \item $g_1,\dots,g_{\binom{d+m-2}{m-2}} \in K[u_1,\dots,u_m]_{d}$.
    \end{itemize}
Note that all of them belong to $\Ann_R(F)$. Therefore, it holds $I(Z_F) \subset \Ann_R(F)$. Now, on the one hand, $\dim [R/I(Z_F)]_2=\dim  [A_F]_2+\binom{d+m-3}{m-2}$ (\ref{HFRI} and Lemma \ref{h2GeneralFullPerazzo}) and $\Ann_R(F)$ has $\binom{d+m-1}{m-1}-\binom{d+m-2}{m-2}=n+1$ more generators of degree $d$ than $I(Z_F)$ does. On the other hand, by Proposition \ref{bettiNumbersIZF}, we know a minimal free resolution of $R/I(Z_F)$:
$$0\rightarrow \begin{array}{c}
R(-n-m-1)^{\beta_{n+m,n+m+1}}\\
\oplus \\
R(-n-m-d+1)^{\beta_{n+m,n+m+d-1}}
\end{array}
\rightarrow \cdots  
\rightarrow \begin{array}{c}
R(-2)^{\beta_{1,2}} \\
\oplus \\
R(-d)^{\beta_{1,d}}
\end{array}
\rightarrow 
I(Z_F) \rightarrow 0
$$
where
$$
\beta_{i,j}=
\begin{cases}
    \beta_{i,d-1+i} =  \sum_{j=0}^{m-2} \Bigl( \binom{j+n+1}{i-1}\sum_{\ell=0}^{d-1}\binom{\ell+j-1}{j-1}\Bigr) \\\\
    \beta_{i,i+1}= i\binom{n+2}{i+1} + (n+1)\Bigl( \binom{n+m}{i} - \binom{n+1}{i} \Bigr),
\end{cases}
$$
as well as a minimal free resolution of the canonical module $\omega_{R/I(Z_F)}$:
$$
0\rightarrow  R(-n-m-1) \rightarrow 
 \begin{array}{c}
R(-2)^{\beta_{1,d} }\\
\oplus \\
R(-d)^{\beta_{1,2}}
\end{array}
\rightarrow \cdots \rightarrow
\begin{array}{c}
R(d-2)^{\beta_{n+m,n+m+d-1}} \\
\oplus \\
R^{\beta_{n+m,n+m+1}}
\end{array}
\rightarrow 
\omega _{R/I(Z_F)}\rightarrow 0,
$$
which means that $\omega_{R/I(Z_F)}(-d)$ has $\beta_{n+m,n+m+1} = n+1$ generators of degree $d$ and $\beta_{n+m,n+m+d-1}=\binom{d+m-3}{m-2}$ generators of degree $2$. In consequence, we can construct the exact sequence
\begin{equation}\label{generalDoublingSequence}
        0 \rightarrow \omega_{R/I(Z_F)}(-d) \xrightarrow{\psi} R/I(Z_F) \rightarrow R/\Ann_R(F) \rightarrow 0.
\end{equation}
Therefore, $A_F$ is a doubling of $Z_F$.

\vskip 2mm
 (2) We have just seen that  $A_F$ is a doubling of a 0-dimensional scheme $Z_F\subset \PP^{n+m}$  consisting of $n+1$ double points in $\PP^{m-1}$. Using the mapping cone process in (\ref{generalDoublingSequence}) one  computes a  minimal free resolution of $A_F$. Since $A_F$ is a Gorenstein ring, its minimal free resolution is self-dual up to twist. Therefore, we only need to compute the graded Betti numbers of the $1$-st non-zero row of its Betti table. For $2\le i \le n+m+1$, let $\beta_{i-1,i}$ and $\beta_{i-1,d+i-2}$ be the graded Betti numbers of the $1$-st and (d-1)-th rows of the Betti table of $R/I(Z_F)$, respectively. Then, $\alpha_i = \beta_{i-1,i} + \beta_{n+m+2-i,n+m+1+d-i}$, which proves what we want.
\end{proof}

\begin{example} 
Consider the Perazzo algebra $A_F$ whose Macaulay dual generator is the one studied by Stanley in \cite{s2}, that is,
$$
F=X_0U^3+X_1V^3+X_2W^3+X_3U^2V+X_4U^2W+X_5V^2U+X_6V^2W+X_7W^2U+X_8W^2V+X_9UVW.
$$
Firstly, and after a change of variables, $A_F$ is the doubling of the 0-dimensional subscheme $Z_F\subset \PP^{12}$ with homogeneous ideal 
$$
  I(Z_F) =\bigcap_{i=0}^9 \left(x_0,\dots,x_i^2,\dots,x_9,\lambda_{i,v}u-v,\lambda_{i,w}u-w\right)\subset K[x_0,\dots,x_9,u,v,w],
$$
where $Z_F$ is the union of $10$ double points, all of them with support on $x_0=x_1=\cdots=x_9=0$. By Proposition \ref{bettiNumbersIZF}, we know the graded Betti numbers of  $I(Z_F)$ and, by duality, we also know the graded Betti numbers of $\omega_{R/I(Z_F)}$. Therefore, using the exact sequence
$$
0\longrightarrow \omega_{R/I(Z_F)}\longrightarrow R/I(Z_F)\longrightarrow A_F\longrightarrow 0
$$ 
and applying the mapping cone process we get  a minimal free resolution of $A_F$ and the Betti table of $R/\Ann_R(F)$:

\vskip 2mm
    \begin{center}
    \begin{tabular}{ r | c c c c c c c c c c c c c c}
           &  0 &  1 &  2 & 3 & 4 & 5 & 6 & 7 & 8 & 9 & 10 & 11 & 12 & 13\\
    \hline
        0 &  1 & · &  · &  · &  · & · & · & · & · & · & · & · & · & · \\  
        1 &  · & 79 & 585 & 2220 & 5403 & 9150 & 11178 & 9975 & 6470 & 2979 & 925 & 174 & 15 & ·  \\
        2 &  · & · &  · &  · & · & . & · & · & · & · & · & · & · & · \\
        3 &  · & 15 & 174 & 925 & 2979 & 6470 & 9975 & 11178 & 9150 & 5403 & 2220 & 585 & 79 & · \\
        4 &  · & · & · & · & · & · & & · & · & · & · & · & · & 1
    \end{tabular}
    \end{center}
\end{example}

\begin{remark} 
Theorem \ref{mainTheoremGeneralFullPerazzos} can be generalized to many other Perazzo algebras not necessarily full with dual generator 
$$
F=X_0L_0^{\gamma-1}+\cdots +X_nL_n^{\gamma-1}\in K[X_0,\dots,X_d,U_1,\dots,U_m],
$$
where $L_i=U_1+\lambda_{i,2}U_2\cdots+\lambda_{i,m}U_m$, with $\lambda_{i,k}\not=\lambda_{j,k}$ for any $k\in\{2,\dots,m\}$, $i,j\in\{0,\dots,n\}$ and $\gamma > n+1$. For example, assume $m=2$ and consider a Perazzo algebra $A_F$ with dual generator $F=X_0L_0^4+X_1L_1^4+X_2L_2^4+X_3L_3^4 \in S_5$. Then, the Betti table of $R/\Ann_R(F)$:
    
    \begin{center}
    \begin{tabular}{ r | c c c c c c c}
           &  0 &  1 &  2 & 3 & 4 & 5 & 6 \\
    \hline
        0 &  1 & · &  · &  · &  · &  . & · \\  
        1 &  · & 14 & 36 & 39 & 20 & 4 & · \\
        2 &  · & 1 &  4 &  6 & 4 & 1 & · \\
        3 &  · & 1 &  4 &  6 & 4 & 1 & · \\
        4 &  · & 4 & 20 & 39 & 36 & 14 & · \\
        5 &  · & · & · & · & · & · & 1
    \end{tabular}
    \end{center}
\end{remark}

\section{Open problems}

In this short last section we state a couple of natural questions/problems stemming from our 
work.

\begin{problem}
   To determine the minimal free resolution or at least the graded Betti numbers of {\em any} Perazzo algebra.
\end{problem}

\begin{problem}
     Is {\em any} Perazzo algebra $A_F$ a doubling of a 0-dimensional scheme $Z_F\subset \PP^{n+m}$? 
     
     If not, can we characterise the homogeneous  forms $F\in K[X_0,\cdots ,X_n,U_1,\cdots ,U_m]$ of degree $d$ such that the Perazzo algebra $A_F$ is a doubling of a suitable 0-dimensional subscheme?
\end{problem}

\subsection*{Data availability statement} Data sharing not applicable to this article as no datasets were generated or
analysed during the current study.

\section*{Declaration}

\subsection*{Funding} The first author has been partially supported by the grant PID2020-113674GB-I00.
\subsection*{Conflict of interest} On behalf of all authors, the corresponding author states that there is no conflict of
interest.

\end{document}